\documentclass{article}

\usepackage{graphicx} 
\usepackage{amsmath,amsthm,verbatim,amssymb,amsfonts,amscd, graphicx}
\usepackage[utf8]{inputenc}
\usepackage{graphics}
\usepackage{hyperref}
\usepackage[baseline]{euflag}
\usepackage{enumitem}
\usepackage{apptools}
\usepackage{titlesec}
\usepackage{appendix}
\usepackage{xcolor}
\usepackage{dsfont}
\usepackage{appendix}

\title{Conditions for equality and stability in Shannon's and Tao's entropy power inequalities}

\author{Lampros Gavalakis%
	\thanks{Statistical Laboratory, DPMMS,
	University of Cambridge,
	Centre for Mathematical Sciences,
        Wilberforce Road,
	Cambridge CB3 0WB, U.K.
        Email: {\tt lg560@cam.ac.uk}
	}
\and 
Ioannis Kontoyiannis%
	\thanks{Statistical Laboratory, DPMMS,
	University of Cambridge,
	Centre for Mathematical Sciences,
        Wilberforce Road,
	Cambridge CB3 0WB, U.K.
        Email: {\tt yiannis@maths.cam.ac.uk}.
}
}

\date{May 2025}

\theoremstyle{plain}

\newtheorem{theorem}{Theorem}
\newtheorem{corollary}[theorem]{Corollary}
\newtheorem{lemma}[theorem]{Lemma}
\newtheorem{proposition}[theorem]{Proposition}

\newtheorem{remark}[theorem]{Remark}
 
\newcommand{\SI}{49\sigma+21}

\newcommand{\R}{\mathbb{R}}

\renewcommand{\ref}[1]{(\ref{#1})}

\newcommand{\Is}{{\rm Is}}

\newcommand{\gnt}{g_{t_0}^{(n)}}

\newcommand{\Var}[1] {
\mathrm{Var}#1
}

\newcommand{\Z}{
\mathbb{Z}
}

\newcommand{\depi}[1] {
\delta_{\mathrm{EPI},#1}
}
\newcommand{\Xt}{
\tilde{X}^t
}
\newcommand{\Yt}{
\tilde{Y}^t
}

\newcommand{\essinf}{\mathop{\rm ess\, inf}}
\newcommand{\f} {f_{X+U}}
\newcommand{\F}{F_{X+U}}
\newcommand{\xmp}{x_m^{(p)}}

\newcommand{\Xp} {
X^{\prime}
}
\newcommand{\be}{\begin{equation}}
\newcommand{\ee}{\end{equation}}
\newcommand{\Xl}{\sqrt{\lambda} X_1 + \sqrt{1-\lambda} X_2}
\newcommand{\Yl}{\sqrt{\lambda} \tilde{X}^t_1 + \sqrt{1-\lambda} 
\tilde{X}^t_2}

\begin{document}

\maketitle

\begin{abstract}
We show that there is equality in Shannon's Entropy Power Inequality (EPI) if and only if the random variables 
involved are Gaussian, assuming nothing beyond the existence of differential entropies. This is done by 
justifying de Bruijn's identity without a second moment assumption. Part of the proof also relies on a
re-examination of an example of Bobkov and Chistyakov (2015), which shows that there exists a random variable $X$ with finite differential entropy $h(X),$ such that $h(X+Y) = \infty$ for any independent random variable $Y$ with finite entropy. We prove that either $X$ has this property, or $h(X+Y)$ is finite for any independent $Y$ that does not have this property. Using this, 
we prove the continuity of  $t \mapsto h(X+\sqrt{t}Z)$ at $t=0$, where $Z \sim \mathcal{N}(0,1)$ is independent of $X$, under minimal assumptions. We then establish two stability results: A qualitative stability result for Shannon's EPI in terms of weak convergence under very mild moment conditions, and a quantitative stability result in Tao's discrete analogue of the EPI under log-concavity. The proof for the first stability result is based on a compactness argument, 
while the proof of the second uses the Cheeger inequality and leverages concentration properties of discrete log-concave distributions. 
\end{abstract}

\footnotetext{This work was supported in part by the EPSRC-funded INFORMED-AI project EP/Y028732/1.}

\section{Introduction and main results}
\subsection{Shannon's Entropy Power Inequality}
One of the equivalent formulations of Shannon's celebrated entropy power inequality (EPI) states that, for any pair of independent random variables $X,Y$ in $\mathbb{R}$ and any $\lambda \in (0,1),$
\begin{equation} \label{EPI}
h(\sqrt{\lambda}X + \sqrt{1-\lambda}Y) \geq \lambda h(X) + (1-\lambda)h(Y)
\end{equation}
where $h(X)$ denotes the (differential) entropy of a random variable $X$ having density $f$, 
\begin{equation}
h(X) =h(f)= -\int_{\mathbb{R}}{f(x)\log{f(x)}dx},
\label{eq:h}
\end{equation}
with the  convention that we set $h(X) = -\infty$ whenever the integral in~(\ref{eq:h}) does not exist, including in the case when $X$ does not have a density with respect to Lebesgue measure on $\R$. In particular, when we say that the differential entropy $h(X)$ of some random variable $X$ is finite, we mean that the integral exists and $-\infty<h(X)<\infty $, while ``$h(X) < \infty$" 
also includes the cases when the integral does not exist and when it does exist and equals $-\infty$.
[Logarithms are natural logarithms throughout.]

The first proof of the EPI, after it was stated in Shannon's paper \cite{shannon:48}, was given by Stam \cite{stam:59} in 1959 and made use of de Bruijn's identity,
\be \label{debruijn}
\frac{d}{dt}h(X+\sqrt{t}Z) = \frac{1}{2}I(X+\sqrt{t}Z), \quad t>0.
\ee
Here, $Z$ is a standard Gaussian independent of $X$, and $I(X)$ denotes the Fisher information of $X$,
where
\be \label{FIdef}
I(X) = \int_{\R}{\frac{\bigl(f^{\prime}(x)\bigr)^2}{f(x)}dx},
\ee
whenever $X$ has an absolutely continuous density $f$;
otherwise, we  set $I(X) = \infty$, c.f.~\cite{bobkov:22}. 

Stam's proof was revisited by Blachman~\cite{blachman:65} and it was also presented in the review of Dembo, Cover and Thomas \cite{dembo-cover-thomas:91}.
Using de Bruijn's identity \eqref{debruijn}, the EPI is derived as a consequence of the convolution inequality for Fisher information, also known as Stam's inequality \cite{blachman:65,stam:59}:
$$
I(X+Y)^{-1} \geq I(X)^{-1}+I(Y)^{-1},
$$
where $X,Y$ are arbitrary independent random variables. 
The proof of de Bruijn's identity \eqref{debruijn} requires justification of exchange of differentiation and expectation.
As has been pointed out before \cite{klartag:25, rioul:17}, this was only  done in 1984 by Barron \cite{barron:cltTR}, under the assumption that $X$ has finite variance.  Around the same time, the Bakry-Em{\'e}ry theory was developed \cite{bakry-emery:85}, where a related representation for the time derivative of the relative entropy $D(P_t\|Q)$ between the time-$t$ distribution $P_t$ of a diffusion and its invariant measure $Q$ was obtained. 
A proof of the EPI using de Bruijn’s identity also appeared in a slightly more general form in Carlen and Soffer \cite{carlen:91}, where a qualitative stability result is obtained as well. We will discuss this in more detail in Section \ref{introstabilitysec}. More recently, another proof of the EPI in the same spirit, using a closely related derivative identity for the mutual information, was given by Verd\'{u} and Guo~\cite{verdu:06}, again assuming finite variance.

One of the main contributions of the present work is the justification of de Bruijn's identity~\eqref{debruijn} under minimal assumptions and in particular under no moment conditions, see Theorem~\ref{ourdebruijn}. 

Another part of Stam's proof which also uses the finite variance assumption, is the continuity of $t \mapsto h(X+\sqrt{t}Z)$ at $t=0$. We will show in Theorem \ref{entropycontTh} that this can be justified without moment conditions but, as we will see, this continuity fails without mild additional assumptions.

A different classical proof of the EPI is due to Lieb \cite{lieb:78} (see also \cite{dembo-cover-thomas:91}), using the sharp form of Young's inequality, which yields a family of inequalities for R{\'e}nyi entropies. However, this proof requires integrability of some power of the densities, i.e., $\int{f^{\alpha}} < \infty$ for some $\alpha>1$, in order to obtain the EPI for differential entropy in the limit as $\alpha \to 1$. One can remove this condition to obtain the EPI for any pair of random variables using a truncation argument as in Bobkov and Chistyakov \cite{bobkov:15}. But because of this truncation, this method of proof does not settle the equality case without the above integrability assumption. 

A number of other proofs of the EPI have appeared in the literature, under 
a variety of different assumptions. A simple proof, which also generalizes to give entropy monotonicity along the central limit theorem (CLT), is due to Madiman and Barron \cite{madiman:07}, under the finite variance assumption. Several different proofs 
were proposed by Rioul \cite{rioul:07, rioul:11, rioul:17b,rioul:17}. 
In particular, the equality case is settled in \cite{rioul:17}, 
assuming (among other conditions) that the densities of $X$ and $Y$ 
are everywhere positive and continuous. 

Szarek and Voiculescu \cite{szarek:00} gave a nice, geometric and intuitive proof using the Brunn-Minkowski inequality and a rearrangement inequality by Brascamp, Lieb and Luttinger. Due to the limiting nature of this argument, the equality case is not settled there either. 
Anantharam, Jog and Nair \cite{anantharam:22} obtained a proof the unifies the EPI and the Brascamp-Lieb inequality using a doubling trick, assuming finite second moments. 
Courtade \cite{courtade:16b} also provided a simple proof for the monotonicity of entropy along the CLT using the maximal correlation, assuming finite variances and smooth densities. 
Finally, a recent, remarkable proof using the F{\"o}llmer process is due to Lehec \cite{lehec:13}, which assumes finite second moments. A similar idea was used by Eldan and Mikulincer \cite{eldan:20} to obtain stability under log-concavity. We will revisit this topic in Section \ref{stabilitysec}. 

Since the proofs in the literature only settle the equality case of the EPI in restricted cases (assuming either finite second moments or under regularity assumptions on the densities of $X$ and $Y$), it is natural to ask: 
{\em Assuming only finiteness of $h(X)$ and $h(Y)$, does equality in the EPI \eqref{EPI} hold if and only if  $X$ and $Y$ are Gaussian?}
We answer this question in the affirmative, using our general version of de Bruijn's identity, a result which may also be of independent interest. 

Before stating this in Theorem~\ref{ourEPI} below, we note that, 
as Bobkov and Chistyakov \cite{bobkov:15} point out, one needs to be careful even when formulating the EPI \eqref{EPI}: there exist random variables $X,Y$ such that the integrals in the definitions of $h(X)$ and $h(Y)$ exist, while $h(X+Y)$ does not. Under the convention that $h(X) = -\infty$ when the integral does not exist,  \eqref{EPI} holds true for any pair of independent random variables $X,Y.$ 

Furthermore, Bobkov and Chistyakov showed that the following ``discontinuity'' phenomenon may occur \cite[Proposition V.8]{bobkov:15}:
\begin{proposition}[\cite{bobkov:15}] \label{bobkovCE}
There exists a random variable $X$ such that $-\infty < h(X) < \infty$, while $h(X+Y) = \infty$ for every independent random variable $Y$ with $h(Y) > -\infty$.
\end{proposition}
This example shows that $t \mapsto h(X+\sqrt{t}Z)$, where $Z$ is a standard Gaussian independent of $X$, need not be continuous at $t=0$. Establishing continuity at $t=0$ is a key step in Stam's proof of the EPI, and it fails when no conditions other than the existence of $h(X), h(Y)$ are imposed. For the same reason, de Bruin's identity \eqref{debruijn} may also fail, as its right-hand side is finite. 

Our first main contribution, Theorem \ref{Thtwopossibilities} below, shows that the family of counterexamples provided by Bobkov and Chistyakov above is the ``only" pathological case. 
For convenience, we first define the class of random variables having the property of the counterexample of Proposition \ref{bobkovCE}:
$$
\mathcal{C}_{\rm BC} := \{X: h(X) \text{ is finite and } h(X+Y) = \infty \text{ for any independent } Y \text{ with $h(Y) > -\infty$} \}.
$$
\begin{theorem} \label{Thtwopossibilities}
 Let $X$ be a random variable in $\mathbb{R}$ whose differential entropy exists and is finite. There are only two possibilities: 
 \begin{enumerate}
\item Either $X\in \mathcal{C}_{\rm BC}$, or 
\item $h(X+Y) < \infty$ for any independent random variable $Y \notin \mathcal{C}_{\rm BC}$ with $h(Y)<\infty$.
 \end{enumerate}
\end{theorem}

In addition, we show that $X \notin \mathcal{C}_{\rm BC}$ is enough to justify the continuity of entropy under Gaussian perturbation: 
\begin{theorem} [Entropy continuity under Gaussian perturbation] \label{entropycontTh}
Let $X$ be a random variable in $\R$ with finite differential entropy, for which there exists an independent random variable $Y$ with finite differential entropy such that $h(X+Y) < \infty$. 
Then, if $Z \sim \mathcal{N}(0,1)$ is independent of $X$,
$$
\lim_{t\downarrow 0}{h(X+\sqrt{t}Z)} = h(X). 
$$
\end{theorem}

The proofs, given in Section \ref{debruijnsec}, exploit a submodularity-for-sums inequality for entropy and use a truncation argument, together with some uniform estimates on the mutual information, and a few careful applications of dominated convergence. 

Our next goal is to derive de Bruijn's identity \eqref{debruijn} under the weakest possible conditions. As mentioned, we see from Proposition \ref{bobkovCE} that the existence and finiteness of $h(X)$ is not enough. Nevertheless, we show:
\begin{theorem}[De Bruijn's identity without finite variance] \label{ourdebruijn}
Let $X$ be a random variable in $\R$ with $h(X) < \infty$, for which there exists an independent random variable $Y$ with finite differential entropy such that $h(X+Y)$ exists and is finite. 
Then, if $Z \sim \mathcal{N}(0,1)$ is independent of $X$,
$$\frac{d}{dt}h(X+\sqrt{t}Z) = \frac{1}{2}I(X+\sqrt{t}Z),\quad t>0.$$
\end{theorem}

Theorem~\ref{ourdebruijn} is proved in Section \ref{debruijnsec}.
The key idea is to replace the step in the proof of Barron \cite{barron:cltTR} where the finite variance assumption is used, with the finiteness of $h(X+\sqrt{t_0}Z)$ for some $t_0>0$ (see Lemma \ref{lemmaexchange}), which is a consequence of the submodularity-for-sums inequality for entropy. A data processing argument together with an application of dominated convergence complete the proof. 

Thus, we are able to prove the following general version of the EPI. We require no assumptions for the EPI to hold and we settle the equality case under the minimal assumption that the entropies exist and are finite. Clearly, if both sides are infinite, equality can be achieved trivially without the random variables being Gaussian. 
\begin{theorem}[EPI] \label{ourEPI}
Let $X,Y$ be independent random variables on $\R$. Then, for any $\lambda\in(0,1),$
\begin{equation} 
h(\sqrt{\lambda}X + \sqrt{1-\lambda}Y) \geq \lambda h(X) + (1-\lambda)h(Y).
\end{equation}
If $h(X)$ and $h(Y)$ exist and are finite, there is equality if and only if $X$ and $Y$ are Gaussian with the same (finite) variance. 
\end{theorem}

To the best of our knowledge, this is the first characterization of the equality case in the EPI under no assumption other than existence of entropies. 
Theorem~\ref{ourEPI} is proved in the Appendix.
For completeness, we present the complete, general version of Stam's proof, 
relying on Theorems \ref{entropycontTh} and \ref{ourdebruijn}. 

As noted above, a version of de Bruijn's identity for the derivative of relative entropy can be established using the Bakry-{\'E}mery theory \cite[Proposition 5.2.2]{bakry:book}. For a diffusion $\{X_t\}$ on $\R$, this is proved for every function $f$ in the associated Dirichlet domain $\mathcal{D}(\mathcal{E})$. The present setting corresponds to the Ornstein-Uhlenbeck process on the real line, with $f = \frac{d\mu}{d\gamma}$, where $\mu$ denotes the law of $X$ and $\gamma$ is a Gaussian measure. 
Our assumptions are strictly weaker: here, the Dirichlet domain reduces to the set of functions having a weak derivative in $L^2(\gamma)$ \cite[p.~126]{bakry:book}, which excludes, for example, functions with jumps. Our result does not require any such regularity assumptions and includes, for example, the case where $X$ has a piecewise constant density. We note also that, since we are not assuming finite variance for $X$, it is not clear how to translate 
the derivative of relative entropy 
in \cite[Proposition 5.2.2]{bakry:book} 
to the derivative of the entropy itself.

Finally, we state two simple consequences of Theorem \ref{ourdebruijn}. Sometimes, in applications arising in the study of generative models in machine learning, it is of interest to differentiate the entropy of a diffusion with discrete initial distribution. Corollary \ref{discretede} says that de Bruijn's identity is valid for any discrete $X$ with finite discrete entropy $H(X)$. Recall that the discrete (Shannon) entropy of a discrete random variable $X$ with probability mass function $p$ on a finite or countably infinite set $A$ is
\begin{equation}
H(X)=H(p)=-\sum_{x\in A}p(a)\log p(a).
\label{eq:H}
\end{equation}
\begin{corollary} \label{discretede}
Let $X$ be a discrete real-valued random variable with finite entropy $H(X)$. Then, if $Z \sim \mathcal{N}(0,1)$ is independent of $X$,
$$\frac{d}{dt}h(X+\sqrt{t}Z) = \frac{1}{2}I(X+\sqrt{t}Z),\quad t>0.$$
\end{corollary}
\begin{proof}
By \cite[Lemma 5.1]{bobkov:20b} we have $h(X+Z) \leq H(X) + h(Z) < \infty$.
By applying Theorem \ref{ourdebruijn} to $X+\sqrt{\varepsilon}Z^{\prime}$ for some fixed $0<\varepsilon < t$, where $Z^{\prime}$ is a standard Gaussian independent of $X$ and $Z$, we obtain 

$$
\frac{d}{dt}h(X+\sqrt{t}Z) = \frac{d}{dt}h(X +\sqrt{\varepsilon}Z^{\prime} +\sqrt{t-\varepsilon}Z)  = \frac{1}{2}I(X +\sqrt{\varepsilon}Z^{\prime} +\sqrt{t-\varepsilon}Z) = \frac{1}{2}I(X+\sqrt{t}Z).
$$
\end{proof}

A combination of Theorems \ref{entropycontTh} and \ref{ourdebruijn} with the mean value theorem gives de Bruijn's identity at $t=0$. Recall that, by convention, we set the Fisher information $I(X)=\infty$ for any random variable that does not have an absolutely continuous density.

\begin{corollary} \label{debruijn0}
Let $X$ be a random variable in $\R$ with $h(X) < \infty$, for which there exists an independent random variable $Y$ with finite differential entropy such that $h(X+Y)$ exists and is finite. Then, if $Z \sim \mathcal{N}(0,1)$ is independent of $X,$
\be \label{debruijn0eq}
\frac{d}{dt}h(X+\sqrt{t}Z)\Big|_{t=0^+} = \frac{1}{2}I(X).
\ee
\end{corollary}

\begin{proof}
Corollary 15.3 of \cite{bobkov:22} states that for any random variable $X$,  
\be \label{bobkovsFIconv}
I(X+\sqrt{t}Z) \to I(X) \quad \text{as } t\to 0.
\ee
Let $r(t) := h(X+\sqrt{t}Z), t\geq 0.$ By Theorem \ref{entropycontTh}, $r(t)$ is continuous (from the right) at $t=0$ and by Theorem \ref{ourdebruijn} it is differentiable for any $t>0$ with 
$$
\frac{d}{dt}r(t) = \frac{1}{2}I(X+\sqrt{t}Z).
$$
Therefore, noting that \eqref{bobkovsFIconv} means $r^{\prime}(t) \to \frac{1}{2}I(X)$ as $t \downarrow 0^{+}$, the mean value theorem gives \eqref{debruijn0eq}.
\end{proof}

\subsection{(In)stability of Shannon's EPI} \label{introstabilitysec}
The next question we address is that of stability in Shannon's EPI: if the deficit in the EPI \eqref{EPI} is small, are the random variables $X,Y$ close (in some appropriate sense) to the extremizers, i.e., to Gaussians? 
To make this precise, for $\lambda\in(0,1)$ let 
$$
\depi{\lambda}(X,Y) := h(\sqrt{\lambda}X + \sqrt{1-\lambda}Y) - \lambda h(X) + (1-\lambda)h(Y).
$$
Does
$$
\depi{\lambda}(X,Y) \leq \delta,
$$
imply that
$$
d(X, {G}_X), d(Y, \text{G}_Y) \leq \epsilon(\delta), 
$$
for some Gaussians $G_X,G_Y$, some ``distance'' measure $d$, and with $\epsilon(\delta) \to 0$ as $\delta \to 0$?
Is it possible to obtain quantitative estimates for $\epsilon(\delta)$?

The first (positive) stability result was obtained by Carlen and Soffer in \cite{carlen:91}, where it was shown that, among isotropic random vectors that satisfy a uniform bound on the Fisher information and a uniform bound on the decay of the second moment tails, $\delta_{\mathrm{EPI},\lambda} \to 0$ implies convergence to Gaussianity in relative entropy.
 In the reverse direction,
 Courtade, Fathi and Pananjady \cite{courtade:18} constructed a sequence of random variables that satisfy $\depi{\lambda} \to 0,$ while their relative entropy from any Gaussian is bounded away from zero. These random variables do not satisfy Carlen and Soffer's uniform bound on the second moment tails, but they have unit variance and bounded Fisher information. 
A quantitative stability result for uniformly log-concave random variables
is also established in \cite{courtade:18}.

 Ball, Barthe and Naor \cite{ball:03} proved that, if $X$ has finite Poincar{\'e} constant $C_P(X)$ (see Section \ref{taosection}), then
\begin{equation} \label{eq:ball}
D(X) \leq \frac{2C_P(X) + 2\sigma^2}{\sigma^2}\depi{\frac{1}{2}}(X,X),
\end{equation}
where 
$\depi{\frac{1}{2}}(X,X)$ refers to the case when $X$ and $Y$
have the same distribution,
$\sigma^2$ is the variance of $X$, and $D(X)$ denotes the relative entropy between $X$
and a Gaussian with the same mean and variance as $X$.
This was generalized to higher dimensions under the extra 
assumption that $X$ is log-concave by Ball and Nguyen \cite{ball:12}.
Related bounds were also established in \cite{KM:14,KM:16};
in particular,
using the results of Johnson and Barron \cite{johnson-barron:04},
it was shown in \cite{KM:14} that, if $X$ has finite Poincar{\'e} constant $C_P(X)$, then 
\begin{equation} 
\label{kontoyianniseq}
D(X) \leq \frac{2C_P(X) + \sigma^2}{\sigma^2}\depi{\frac{1}{2}}(X,X). 
\end{equation}
Further discussion of the relation between
$\depi{\frac{1}{2}}(X,X)$ and the closely related ``entropic
doubling constant'' is given 
in \cite{gavalakis-doubling:24,gavalakis-rupert:arxiv}.

Eldan and Mikulincer \cite{eldan:20} used an adaptation of Lehec's idea \cite{lehec:13} to generalize \eqref{eq:ball} for any $\lambda$ (instead of only considering $\lambda = \frac{1}{2}$) and to obtain improved constants in the case where the random variables are already close to Gaussian. 
This result played a key role in the recent breakthrough resolution of Bourgain's slicing problem \cite{klartag:24, guan:24}.

Recently, a weak qualitative stability result was established 
in \cite{gavalakis-doubling:24}, in a similar spirit as that 
in \cite{carlen:91}, 
but without the Fisher information assumption, and in terms of 
weak convergence rather than convergence in relative entropy. 

Although the counterexample of \cite{courtade:18} says that stability may fail for strong distances such as relative entropy, it is natural to ask whether we may still have stability in the  sense of weak convergence. The result 
of \cite{gavalakis-doubling:24} does not completely answer this, as the uniform integrability assumption in that result, which is necessary in the proof due to the use of the result of \cite{carlen:91}, is not satisfied by the counterexample of \cite{courtade:18}.

Our main stability result is only qualitative and in the weakest possible distance, the L{\'e}vy metric $d_L$ which metrizes weak convergence (see Section \ref{stabilitysec} for definitions), but it only assumes finite $k$th moments for some $k>1$:
\begin{theorem} \label{qualitativestabilityTh}
Suppose $X,Y$ are random variables with finite 
differential entropies, $-\infty <h(X),h(Y) < \infty$, 
and
$$
\mathbb{E}|X|^k,\mathbb{E}|Y|^{k} <\infty \text{ for some }k >1.
$$
Fix $\lambda \in (0,1)$. Then 
\be \label{epsdeltaeq}
\mbox{for each $\epsilon >0$ there is a $\delta > 0$ s.t. }\;
\depi{\lambda}(X,Y) < \delta \; \text{ implies } \; d_{\mathrm{L}}(X,G_1),  d_{\mathrm{L}}(Y,G_2)  < \epsilon,
\ee
for some Gaussian random variables $G_1,G_2$ with the same variance.

Equivalently, suppose $\lambda\in(0,1)$ and
$\{(X_n,Y_n)\}$ is a sequence of pairs of
random variables such 
that, for each $n$, $X_n,Y_n$ are independent with have finite differential 
entropies. If $\{X_n\}$ and $\{Y_n\}$ have uniformly bounded 
$k$th absolute moments for some $k>1$ and $\depi{\lambda}(X_n,Y_n) \to 0$ 
as $n\to \infty$, then the limits of any subsequence along which 
$\{X_n\}$ and $\{Y_n\}$ both converge weakly, are Gaussian with the 
same variance. 
\end{theorem}

Note that the conclusion of Theorem \ref{qualitativestabilityTh} does not imply that the random variables have a weak limit if the deficit in the EPI vanishes;
in fact this may well not hold.
It is also straightforward to check that the counterexample of \cite{courtade:18} satisfies the assumption of our Theorem \ref{qualitativestabilityTh}, since the densities in that example all have unit variance. 

The proof, given in Section \ref{stabilitysec}, is based on a compactness argument. After smoothing the random variables appropriately, the moment condition ensures convergence of differential entropies and we obtain the result by the equality case in the EPI via Theorem~\ref{ourEPI}.

\subsection{Stability of Tao's discrete EPI}

Obtaining discrete analogues of the EPI is a subtle task and it is not 
ever clear what is the most natural discrete version of the continuous EPI. 
Tao \cite{tao:10} proved the following discrete analogue for independent 
and identically distributed (i.i.d.) discrete random variables taking 
values in any torsion-free group: 
\begin{theorem}[\cite{tao:10}]
Let $X_1,X_2$ be i.i.d.\ random variables taking values in a discrete subset of a torsion-free group $G$. Then 
$$H(X_1+X_2) \geq H(X_1) + \frac{1}{2}\log{2} - o(1),$$
as $H(X_1) \to \infty$.
\end{theorem}
Here $H$ denotes the discrete (Shannon) entropy, recall~\eqref{eq:H}. 
This can be seen as a discrete analogue of \eqref{EPI} for $\lambda=\frac{1}{2}$, which by scaling, is the same as 
$$
h(X_1+X_2) \geq h(X_1) +\frac{1}{2}\log{2}. 
$$
It is easy to see that without the $o(1)$ term, the inequality 
$$
H(X_1+X_2) \geq H(X_1) + \frac{1}{2}\log{2},
$$
can fail if $H(X_1)$ is small enough. 

In the case of integer-valued random variables,
finite bounds for the $o(1)$ term were established 
in \cite{gavalakis:24, fradelizi:arxiv24}, along with 
generalizations to sums of more than two random variables,
under the discrete log-concavity assumption: 
\begin{theorem}[{\cite[Theorem 2, case $n=1$]{gavalakis:24}}]
Let $X_1, X_2$ be i.i.d.\ log-concave random variables in $\mathbb{Z}$. 
There is an absolute constant $C$ such that,
\begin{equation} \label{discreteEPI}
H(X_1+X_2) \geq H(X_1) + \frac{1}{2}\log{2} - C\frac{\log{\sigma}}{\sigma},
\end{equation}
provided that $\sigma^2:= \Var{(X_1)}$ is large enough.
\end{theorem}

Our main result in the discrete setting provides a quantitative 
stability estimate of \eqref{discreteEPI}:
\begin{theorem} 
\label{mainthdiscrstable}
 Suppose $X$ has a log-concave distribution on the integers, variance $\Var{(X)} = \sigma^2$ and $X_1, X_2$ are i.i.d.\ copies of $X$. Denote by $Z^{(\Z)}$ a Gaussian with the same mean and variance as $X$, discretized on $\mathbb{Z}$. Assume $\sigma \geq 1547$. Then 
there are absolute constants $C_1, C_2$ such that 
\begin{equation} \label{discretestabilityeq}
D(X\|Z^{(\Z)}) \leq C_1\Bigl( H(X_1+X_2) - H(X_1) - \frac{1}{2}\log{2} \Bigr) + C_2\frac{\log{\sigma}}{\sigma},
\end{equation}
where $D(\cdot\|\cdot)$ denotes the relative entropy.
 \end{theorem}
From the proof of Theorem \ref{mainthdiscrstable} it can be seen that we may take $C_1 = 876489$. An explicit value for $C_2$ may also be extracted, but it will be at least $ 10^{10}$, so we do not pursue this here. We expect that these constants can be improved significantly. 

By a discretized Gaussian $Z^{(\Z)}$
in Theorem \ref{mainthdiscrstable}
(also referred to as quantized Gaussian in \cite{gavalakis:clt}) 
we mean 
$$\mathbb{P}(Z^{(\Z)} = k) = \int_{[k,k+1)}\varphi_{\mu,\sigma^2}(x)dx, 
\quad k \in \mathbb{Z}, 
$$
where $\varphi_{\mu,\sigma^2}$ denotes the Gaussian density
with the same mean $\mu$ and variance $\sigma^2$ as $X$.

Theorem \ref{mainthdiscrstable} says that if the discrete EPI \eqref{discreteEPI} is close to equality,
and the if variance of $X_1$ is large enough (a condition
which is, in general, weaker than the entropy being large), 
then $X_1$ is close in relative entropy to a Gaussian discretized 
on the integers. It can also be viewed as a discrete analogue of 
the continuous stability results of 
Ball, Barthe and Naor \eqref{eq:ball}
and of Kontoyiannis-Madiman \eqref{kontoyianniseq}.

We present the proof of Theorem \ref{mainthdiscrstable} in Section \ref{taosection}. Our strategy is to use \eqref{kontoyianniseq} applied to $X+U$,
where $U\sim {\cal U}[0,1]$, and to approximate the continuous EPI deficit 
by the discrete one using the results of \cite{gavalakis:24}. Note that,
since the density of $X+U$ is piecewise constant, it is not log-concave. 
Nevertheless, we show that it has a finite Poincar{\'e} constant,
bounded by an absolute constant times the variance of $X$: 
\begin{proposition} \label{poincareprop}
Let  $X$ be a discrete log-concave random variable on the integers and let $U$ be a continuous uniform random variable on the unit interval,
independent of $X$. Assume $\sigma^2 := \Var{(X)} \geq 4$. Then 
\[
C_P(X + U) \leq 438244 \cdot \sigma^2,
\]
where $C_P(\cdot)$ denotes the Poincar{\'e} constant. 
\end{proposition}

This bound is analogous to the continuous case, up to the large absolute constant. In the proof we use the Cheeger inequality 
$$
C_P(\nu) \leq \frac{4}{\Is^2(\nu)},
$$
where $\Is(\nu)$ denotes the isoperimetric constant of a probability measure $\nu$ (see Section \ref{cheegersec} for definitions). A useful identity obtained by Bobkov and Houdr{\'e} in the one-dimensional case then allows us to exploit the concentration properties of discrete log-concave measures to bound the isoperimetric constant. 

We mention in passing that bounding the Poincar{\'e} constant (or the inverse of the isoperimetric constant) from above in high dimensions is related to the Kannan-Lov{\'a}sz-Simonovits conjecture \cite{KLS:95}, which states that the Poincar{\'e} constant of any (continuous) log-concave random vector is bounded above by an absolute constant, independent of the dimension. The best known upper bound to date is $C \sqrt{\log{n}}$ due to Klartag \cite{klartag:24b}. Of course, our Proposition \ref{poincareprop} is only in dimension $1$. 

Another motivation for considering estimates as in \eqref{discretestabilityeq} comes from the recent proof of Marton's conjecture \cite{GGMT:25, green:25}. The equivalent entropic formulation of the polynomial Freiman-Ruzsa (PFR) conjecture \cite[Theorem 1.8]{GGMT:25} can be seen as a stability estimate in the lower bound 
$$
H(X_1+X_2) -\frac{1}{2}H(X_1) - \frac{1}{2}H(X_2) \geq 0
$$
when $X_1,X_2$ take values in $\mathbb{F}_2^n$, in terms
of a ``distance" known as Ruzsa distance. The extremizers in this case 
are uniform distributions on finite subgroups. 
Obtaining analogous estimates in torsion-free groups is related to PFR over $\mathbb{Z}$, which remains open.  

\section{Proof of de Bruijn's identity} \label{debruijnsec}

In this section we prove Theorems \ref{Thtwopossibilities}, \ref{entropycontTh} and \ref{ourdebruijn}, which provide the necessary tools for the proof
of Shannon's EPI in Theorem~\ref{ourEPI}, including the characterization 
of the case of equality, under no assumptions beyond the existence
of entropies.

We are going to make repeated use of the following submodularity-for-sums inequality, which is a consequence of the data processing inequality for mutual information:
\begin{lemma}[\cite{madiman-itw:08,KM:14}] \label{km14Lemma}
Let $X, Y, Z$ be independent random variables such that $h(Y),h(X+Y),h(Y+Z)$ are all finite. Then  
\be \label{ruzsa}
h(X+Y+Z)  \leq h(X+Y) + h(Y+Z) - h(Y).
\ee
\end{lemma}

In \cite{KM:14} it is assumed that all entropies that appear exist and are finite. However, a careful examination of the proof of Lemma \cite{KM:14} shows that finiteness of $h(X+Y+Z)$ is not necessary. This is because the data processing inequality holds without finiteness assumptions and the expression $I(X;Y) = h(Y) -h(Y|X)$ is valid at least as long $-\infty < h(Y|X) < \infty$. In particular, if $h(Y), h(X+Y)$ and $h(Y+Z)$ are all finite, then $h(X+Y+Z)$ is finite as well. Inequality \eqref{ruzsa} was first observed in the discrete setting in \cite{KV:83}.

\begin{lemma} \label{gaussfinite}
Let $X$ be a random variable in $\mathbb{R}$ with density $f$ and finite 
differential entropy $-\infty < h(X) < \infty$. If there is a random variable 
$Y$, independent of $X$, with finite differential entropy and with $h(X+Y) < \infty$, then 
$$
h(X+Z) < \infty,
$$
for every random variable $Z$ independent of $X$ with finite variance. 
\end{lemma}
\begin{proof}
First we note that if $X,X^{\prime}$ are i.i.d., then we have by Lemma \ref{km14Lemma}
$$
h(X+\Xp) \leq h(X+Y+\Xp) \leq 2h(X+Y) - h(Y) <\infty.
$$

Now we define the binary random variable 
$$I_n := \mathbb{I}_{\{|X| \leq n\}},$$
where we choose $n$ such that 
$p_n := \mathbb{P}(I_n=1)$ is strictly between 0 and 1.
We also let $X_n$ be distributed as~$X$ conditioned on $\{I_n = 1\}$.
Writing $I(V;W)$ for the mutual information \cite{cover:book2}
between two random variables $V$ and $W$,
we observe that $I(X;I_n)$ can be expressed 
\cite{pinsker:book}
in two different ways as 
\begin{align} 
\label{mutualinfoexpr}
I(X;I_n) &= H(I_n) \\
&= h(X) - h(X|I_n),
\end{align}
where $h(\cdot|\cdot)$ denotes the conditional
differential entropy~\cite{cover:book2}, and
$H(I_n) \in [0,1]$.
Thus, we must have
$$
-\infty < h(X|I_n) <\infty.
$$
Therefore, writing $h(W|A)$ for the differential entropy
of a random variable with the distribution of $W$ conditioned
on the event $A$, and
using $p_nh(X_n) + (1-p_n)h(X|I_n=0) = h(X|I_n)$,
we see that 
$$
h(X_n) > -\infty.
$$

Finally, since conditioning reduces entropy, we have 
\begin{align*}
\infty > h(X+\Xp) \geq h(X+\Xp|I_n) &= p_n h(X_n+\Xp) + (1-p_n)h(X+\Xp|I_n = 0) \\
&\geq p_n h(X_n+\Xp) + (1-p_n)h(\Xp) 
\end{align*}
and therefore 
$$
h(X_n+\Xp) <\infty.
$$

Putting everything together, and using Lemma \ref{km14Lemma}, we have 
$$
h(\Xp+Z) \leq h(\Xp+X_n+Z) \leq h(\Xp+X_n) + h(X_n+Z) - h(X_n) < \infty
$$
where $h(X_n+Z)$ is finite since $X_n + Z$ has finite variance 
(as $X_n$ is bounded). 
\end{proof}

In order to prove Theorem \ref{Thtwopossibilities} we only need to generalize Lemma \ref{gaussfinite} to all random variables $Z \notin \mathcal{C}_{\rm BC}$ (with finite entropy), rather than $Z$ having finite variance:

\begin{proof}[Proof of Theorem \ref{Thtwopossibilities}]
Let $X$ be a random variable with finite differential entropy. We need to show that if there exists a random variable $W$ with finite differential entropy itself, independent of $X$, such that 
$$h(X+W) <\infty,$$
and if $Y,V$ are two independent random variables, independent from $X, W$, with finite differential entropies, such that 
$$h(Y+V) < \infty,$$
then 
$$
h(X+Y) < \infty.
$$

Let $Z$ be a standard Gaussian. By another application of Lemma \ref{km14Lemma}

$$
h(X+Y) \leq h(X+Z+Y) \leq h(X+Z) + h(Z+Y) - h(Z) <\infty,
$$
where the finiteness of the last two entropies of the sums
follows by Lemma \ref{gaussfinite}.
\end{proof}

\begin{lemma} \label{truncationlemma}
Let $X$ be a random variable with finite entropy, $Z$ a standard Gaussian independent of $X$, and $t_0 >0$. Assume that $h(X+\sqrt{t_0}Z) < \infty$.
Then 
$$
\lim_{n \to \infty} h(X+\sqrt{t_0}Z\mid|X| \leq n) = h(X+\sqrt{t_0}Z)
$$
\end{lemma}
\begin{proof}
Write $f$ for the density of $X$ and $\gnt$ for the density of $X+\sqrt{t_0}Z$
conditioned on $\{|X| \leq n\}$. Then
$$
\gnt(z) = \frac{1}{p_n}\int_{|x|\leq n}{f(x)\varphi_{t_0}(z-x)dx},
\quad z\in\R,
$$
where $\varphi_t$ denotes the ${\cal N}(0,t)$ density,
and
$p_n = \mathbb{P}{(|X|\leq n)} \to 1$ as $n\to \infty$.
Note also that $p_n\gnt \uparrow g_{t_0}$ as $n\to \infty$ pointwise, where $g_{t_0} = f*\varphi_{t_0}$. 

Using $p_n \to 1$, we have 
\begin{align} \nonumber
  &h(X+\sqrt{t_0}Z\mid|X| \leq n) \\
	\nonumber
  &= -\int_{0\leq g_{t_0}\leq \frac{1}{e}}{\gnt(z)\log{\gnt(z)} dz} -\int_{ \frac{1}{e}<g_{t_0} \leq 1}{\gnt(z)\log{\gnt(z)} dz} -\int_{1<g_{t_0}}{\gnt(z)\log{\gnt(z)} dz} \\ \nonumber
  &= -\frac{1}{p_n}\int_{0\leq g_{t_0}\leq \frac{1}{e}}{p_n\gnt(z)\log{(p_n\gnt(z))} dz} -\frac{1}{p_n}\int_{ \frac{1}{e}<g_{t_0} \leq 1}{p_n\gnt(z)\log{(p_n\gnt(z))} dz}  \\ \label{intsplit}
  &\quad-\frac{1}{p_n}\int_{1<g_{t_0}}{p_n\gnt(z)\log{(p_n\gnt(z))} dz} + o(1).
\end{align}
We consider the three integral terms in \eqref{intsplit} separately. 

In the range $0\leq g_{t_0}\leq \frac{1}{e}$, using the fact that $-x\log{x}$ is non-decreasing for $x\leq \frac{1}{e},$ we have 
$$-p_n\gnt\log{(p_n\gnt)} \uparrow -g_{t_0}\log{g_{t_0}}.$$ Therefore, by monotone convergence, 
\begin{equation} \label{firstint}
    -\int_{0\leq g_{t_0} \leq \frac{1}{e}}{p_n\gnt\log{(p_n\gnt)}} \uparrow -\int_{0\leq g_{t_0} \leq \frac{1}{e}}g_{t_0}\log{g_{t_0}}.
\end{equation}
For the second term, note that $\mathrm{Leb}(\{\frac{1}{e}<g_{t_0} \leq 1\}) \leq e $ because of $\int{g_{t_0} = 1}$, where $\mathrm{Leb}$ denotes the Lebesgue measure. In addition, $-x\log{x}$ is bounded for $0\leq x \leq 1.$ Therefore, by bounded convergence
\begin{equation} \label{secint}
-\int_{ \frac{1}{e}<g_{t_0} \leq 1}{p_n\gnt(z)\log{(p_n\gnt(z))} dz} \to -\int_{ \frac{1}{e}<g_{t_0} \leq 1}{g_{t_0}(z)\log{(g_{t_0}(z))} dz}.
\end{equation}
Finally, the third term is 
\begin{align} \nonumber
&-\frac{1}{p_n}\int_{1<g_{t_0}}{p_n\gnt(z)\log{(p_n\gnt(z))} dz} \\ \label{twolastints}
&= -\frac{1}{p_n}\int_{1<g_{t_0}}{\mathbb{I}_{\{p_n\gnt(z) \leq 1\}}p_n\gnt(z)\log{(p_n\gnt(z))} dz}  -\frac{1}{p_n}\int_{1<g_{t_0}}{\mathbb{I}_{\{p_n\gnt(z) > 1\}}p_n\gnt(z)\log{(p_n\gnt(z))} dz}.
\end{align}
Now, since $p_n\gnt \to g_{t_0}$ pointwise, the integrand of the 
first term in \eqref{twolastints} converges pointwise to $0$. Using a 
similar argument as in \eqref{secint}, we see that that integrand is 
bounded and the integral is over a set of Lebesgue measure at most $1$. 
Hence, by bounded convergence, 
\begin{equation} \label{vanishint}
-\frac{1}{p_n}\int_{1<g_{t_0}}{\mathbb{I}_{\{p_n\gnt(z) \leq 1\}}p_n\gnt(z)\log{(p_n\gnt(z))} dz} \to 0.
\end{equation}
On the other hand, for the second integral in \eqref{twolastints} we have, using that $x\log{x}$ is non-decreasing for $x > 1 > \frac{1}{e}$ and monotone convergence again, 
\begin{equation} \label{finalint}
    \int_{1<g_{t_0}}{\mathbb{I}_{\{p_n\gnt(z) > 1\}}p_n\gnt(z)\log{(p_n\gnt(z))} dz} \uparrow \int_{1<g_{t_0}}{g_{t_0}(z)\log{(g_{t_0}(z))} dz}.
\end{equation}
Substituting \eqref{firstint}, \eqref{secint}, \eqref{vanishint} and \eqref{finalint} into \eqref{intsplit} gives the result. 
\end{proof}

We are ready to give the proof of Theorem \ref{entropycontTh}.
\begin{proof}[Proof of Theorem \ref{entropycontTh}.]
Note first that by Lemma \ref{gaussfinite} we have 
$$
h(X+\sqrt{t_0}Z) <\infty.
$$
for some (in fact for every) $t_0 >0$.

Since $h(X+\sqrt{t}Z) \geq h(X)$, it suffices to show that 
\be \label{sufflimsup}
\limsup_{t\to 0}h(X+\sqrt{t}Z) \leq h(X).
\ee
As in the proof of Lemma \ref{gaussfinite}, define  
$$I_n := \mathbb{I}_{\{|X| \leq n\}},$$ 
let $p_n = \mathbb{P}(I_n=1)$,
and consider the mutual information 
\begin{align} \label{mutualinfosecond}
I(X+\sqrt{t}Z;I_n) = h(X+\sqrt{t}Z) - h(X+\sqrt{t}Z|I_n).
\end{align}
Similarly to \eqref{mutualinfoexpr}, we have
\be \label{uniformint}
0\leq I(X+\sqrt{t}Z;\mathbb{I}_{\{|X|\leq n\}}) \leq H(I_n) \to 0 \quad \text{as } n \to \infty, \text{ uniformly in } t. 
\ee
Therefore, by \eqref{mutualinfosecond}, it suffices to show that,
 for any $\epsilon > 0$.
\be \label{sufftrunc}
\limsup_{t\to 0}h(X+\sqrt{t}Z|I_n) \leq h(X) +\epsilon \quad \text{for all } n \text{ large enough,}
\ee
as this would yield \eqref{sufflimsup}.

Now 
\be \label{condenttrunc}
h(X+\sqrt{t}Z|I_n) = p_nh(X+\sqrt{t}Z\mid |X|\leq n) + (1-p_n)h(X+\sqrt{t}Z\mid|X|>n).
\ee
But for any $t < t_0,$
\be \label{sandwich}
(1-p_n)h(X\mid|X|>n) \leq (1-p_n)h(X+\sqrt{t}Z\mid|X|>n) \leq (1-p_n)h(X+\sqrt{t_0}Z\mid|X|>n). 
\ee
Since also
$$
h(X+\sqrt{t_0}Z|I_n) = p_nh(X+\sqrt{t_0}Z\mid |X|\leq n) + (1-p_n)h(X+\sqrt{t_0}Z\mid|X|>n),
$$
we see by Lemma \ref{truncationlemma} together with \eqref{mutualinfosecond} and \eqref{uniformint} that 
$$
(1-p_n)h(X+\sqrt{t_0}Z\mid|X|>n) \to 0 \quad \text{ (uniformly in } t < t_0).
$$
Furthermore, the lower bound in \eqref{sandwich} also vanishes as $n \to \infty$, uniformly in $t< t_0$, since by an easy application of dominated convergence 
\be \label{enttruncconv}
h(X\mid|X|\leq n) \to h(X),
\ee
and as in \eqref{uniformint},
$$
h(X|I_n) \to h(X). 
$$
Therefore,
$$
(1-p_n)h(X+\sqrt{t}Z\mid|X|>n) \to 0 \quad \text{uniformly in } t<t_0.
$$
Now assuming $n$ is large enough and taking the $\limsup$ as $t\to0$ in \eqref{condenttrunc},
\begin{align}
\limsup_{t\to 0}h(X+\sqrt{t}Z|I_n) &\leq p_n\limsup_{t\to 0} h(X+\sqrt{t}Z\mid |X|\leq n) + \frac{\epsilon}{2} \nonumber\\
&\leq \limsup_{t\to 0} h(X+\sqrt{t}Z\mid |X|\leq n) + \frac{\epsilon}{2} 
\nonumber\\ \label{finitevarianceresult}
&= h(X\mid|X|\leq n) +\frac{\epsilon}{2} \\ \label{lasteentjust}
&\leq h(X) +\epsilon,
\end{align}
where in \eqref{finitevarianceresult} we used that, conditional on the event $\{|X| \leq n\}$, $X$ has finite variance (as it is bounded) and therefore, $h(X+\sqrt{t}Z\mid|X|\leq n) \to h(X\mid|X|\leq n)$ as in 
equations (2.21)--(2.22)
in the proof of Lemma 2.1 of \cite{barron:cltTR}.
In \eqref{lasteentjust} we used \eqref{enttruncconv} and assumed $n$ is large enough.

This establishes \eqref{sufftrunc} and, since $\epsilon>0$ can be taken arbitrary, the proof is complete. 
\end{proof}

As in the proof of Lemma \ref{truncationlemma}, if $X$ has density $f$, 
we write 
\[
g_t = f * \varphi_t, 
\]
for
the density of the convolution of with $\varphi_t$, a Gaussian with 
mean zero and variance $t$. We will need the classical fact that 
$g_t$ always satisfies the heat equation; see, e.g., 
\cite[Lemma 6.2]{barron:cltTR}.

\begin{lemma} 
\label{diffusion}
    Let $X$ be any random variable on $\R$. Then
$$
          \frac{\partial}{\partial t}g_t(y) = \frac{1}{2}\frac{\partial^2}{\partial y^2}g_t(y), \quad \text{ for every } t > 0, \;y \in \R.
$$
\end{lemma}

The next lemma generalizes \cite[Lemma 6.3]{barron:cltTR}, in that we omit 
the finite variance assumption on $X$ and only assume that there exists an 
independent random variable $Y$ such that the entropy of $X+Y$ is finite: 

\begin{lemma} \label{lemmaexchange}
Let $X$ be a random variable with density $f$, finite differential entropy, and for which there exists an independent random variable $Y$ with finite differential entropy itself such that $h(X+Y) < \infty.$ Then, writing as above $g_t = f * \varphi_t$, we have
$$
\frac{\partial}{\partial t}h(X+\sqrt{t}Z) = -\int{\Big(\frac{\partial}{\partial t}g_t(y)\Big)\log{g_t(y)}dy},
\quad t>0.
$$
\end{lemma}

\begin{proof}
By Lemma \ref{gaussfinite} we have 
$$
h(X+\sqrt{t_0}Z) <\infty.
$$
for every $t_0 >0$.
We will show that $|\partial_t g_t \log g_t|$ is dominated by an integrable function uniformly in a neighborhood of $t$. Fix $0 < a < t < t_0/2$.

As in Barron \cite[Eq. (6.12)]{barron:cltTR}, for $t < t_0/2$,
\be \label{partialbound}
|\partial_t g_t(z)| \leq C_{a,t_0} \, g_{t_0}(z),
\ee
for every $z \in \R$.
In addition, $g_t$ satisfies the bounds 
\begin{align*}
g_t(z) &\leq \frac{1}{\sqrt{2\pi t}} \leq \frac{1}{\sqrt{2 \pi a}},
\\ 
\mbox{and}\quad g_t(z) &= \mathbb{E}\Big[\frac{1}{\sqrt{2\pi t}} e^{-\frac{(z-X)^2}{2t}}\Big] 
\geq C_{t_0,a}\mathbb{E}\Big[\frac{1}{\sqrt{2\pi a}} e^{-\frac{(z-X)^2}{2a}}\Big] =C_{t_0,a}\,g_a(z),
\end{align*}
for every $z\in \R$, where $C_{t_0,a} = \sqrt{2a/t_0}.$
Therefore, 
$$
|\log g_t(z)| \leq \frac{1}{\sqrt{2 \pi a}} + C_{t_0,a} + |\log{g_a(z)}|\leq \frac{1}{\sqrt{2 \pi a}} + C_{t_0,a} + |\log g_{t_0}(z)| + \left| \log \frac{g_a(z)}{g_{t_0}(z)} \right|.
$$
Combining with \eqref{partialbound} and writing $C^{\prime}_{t_0,a} = \frac{1}{\sqrt{2 \pi a}} + C_{t_0,a}$ we get 
\be \label{partialloggbound}
|\partial_t g_t(z) \log g_t(z)| \leq  C^{\prime}_{t_0,a}g_{t_0}(z)+   g_{t_0}(z) |\log g_{t_0}(z)| + g_{t_0}(z)\left| \log \frac{g_{t_0}(z)}{g_{a}(z)} \right|.
\ee

The first term in \eqref{partialloggbound} is clearly integrable. The second term is integrable because of $h(X + \sqrt{t_0}Z) < \infty$.
To see that the third term is integrable, we observe first that by the data processing inequality for relative entropy \cite{cover:book2},
    \[
    D(g_{t_0} \| g_{a}) = D(X+\sqrt{t_0}Z\|X+\sqrt{a}Z) \leq D((X,\sqrt{t_0}Z)\|(X,\sqrt{a}Z))=D(\sqrt{t_0}Z\|\sqrt{a}Z) < \infty,
    \]
where the finiteness of the last relative entropy
follows by simple direct calculation. 
This implies that the integral of $g_{t_0}\log[g_{t_0}/g_a]$ is finite,
and Lebesgue integrability implies that the last term in
\eqref{partialloggbound} is also integrable.
[Alternatively, this can be seen directly via the
inequality \cite[pg. 339, proof of Corollary]{barron:clt}
$$
\int_{\R}{f(x)\Big|\log{\frac{f(x)}{g(x)}}\Big|dx} \leq D(f\|g) + [2D(f\|g)]^{\frac{1}{2}},
$$
for any two probability densities $f$ and $g$.]

Thus, $|\partial_t g_t(z) \log g_t(z)|$ is bounded uniformly in $t \in (a, t_0/2)$ by the right-hand side of \eqref{partialloggbound}, which is integrable.

By the mean value theorem and dominated convergence we may exchange 
differentiation and expectation to obtain
\begin{align*}
\frac{\partial}{\partial t}h(X+\sqrt{t}Z) 
&= -\int{\frac{\partial}{\partial t}\bigl(g_t(y)\log{g_t(y)}\bigr)dy} \\
&= -\int\frac{\partial}{\partial t}g_t(y)dy 
- \int{\Big(\frac{\partial}{\partial t}g_t(y)\Big)\log{g_t(y)}dy} \\
&= -\frac{\partial}{\partial t} \int g_t(y)dy - \int{\Big(\frac{\partial}{\partial t}g_t(y)\Big)\log{g_t(y)}dy} \\
&= - \int{\Big(\frac{\partial}{\partial t}g_t(y)\Big)\log{g_t(y)}dy},
\end{align*}
where $\int\frac{\partial}{\partial t}g_t(y)=\frac{\partial}{\partial t} \int g_t(y)$ as in \cite[Eq. (6.10)]{barron:cltTR}.
\end{proof}

Finally, Lemma \ref{lemmabarronFI} below was established 
as Lemma~6.4 in \cite{barron:cltTR}. The same proof as 
in \cite{barron:cltTR} 
works under the present assumptions, 
except we rely on Lemma \ref{lemmaexchange} instead
\cite[Lemma 6.3]{barron:cltTR}. 

Recall the definition of the Fisher information in \eqref{FIdef}. 
Since $X+\sqrt{t}Z$ has a $C^{\infty}$ density, 
the integral in $I(X+\sqrt{t}Z)$ is well-defined and finite. 

\begin{lemma} \cite[Lemma 6.4]{barron:cltTR} \label{lemmabarronFI}
Let $X$ be a random variable with density $f$, finite differential entropy, and for which there exists an independent random variable $Y$ with finite 
differential entropy itself such that $h(X+Y)$ is finite. Then,
$$
I(X+\sqrt{t}Z) = -\int_{\R}\Big(\frac{\partial^2}{\partial y^2}g_t(y)\Big)
\log{g_t(y)} dy, \quad t>0.
$$
\end{lemma}

We are now ready to give the proof of de Bruin's identity without 
the finite second moment assumption: 
\begin{proof}[Proof of Theorem \ref{ourdebruijn}]
By Lemma \ref{gaussfinite}, $h(X+\sqrt{t}Z)$ is finite. The Fisher information $I(X+\sqrt{t}Z) \leq \frac{1}{t}$ is also finite. 
The result then follows from Lemmas \ref{diffusion}, \ref{lemmaexchange} and \ref{lemmabarronFI}. 
\end{proof}

\section{Proof of weak stability in Shannon's EPI} \label{stabilitysec}
This section is devoted to the proof of Theorem \ref{qualitativestabilityTh}. 
Throughout, for a random variable $X$ and $0\leq t\leq 1$, we write
\be \label{pertdef}
\tilde{X}^t = \sqrt{t}X + \sqrt{1-t}Z,
\ee
where $Z$ is a standard Gaussian, independent of $X$.

Recall that for two independent random variables $X_1,X_2$,
$\depi{\lambda}(X_1,X_2)$ was defined
as the difference, $h(\sqrt{\lambda}X_1 + \sqrt{1-\lambda}X_2) - \lambda h(X_1) - (1-\lambda) h(X_2).$

The following lemma shows that the deficit in the EPI decreases under Gaussian perturbation. A version of this lemma was established by Carlen and Soffer \cite{carlen:91}, but under the additional assumption that the random variables have 
covariance matrices with the same (finite) trace. Their proof makes use of an integral form of de Bruijn's identity. Here we make no moment assumptions and provide a simpler argument, which uses only the EPI. We state it for any dimension $d$, but we will only use it for $d=1$.

\begin{lemma} \label{deficitdecrlemma}
Let \( X_1, X_2 \in \mathbb{R}^d\) be independent random
vectors with finite differential entropies, and assume $-\infty < h(X_1+X_2) < \infty$. 
Then for any $\lambda,t \in [0,1]$,
$$
 \depi{\lambda}(\tilde{X}_1^t,\tilde{X}_2^t) \leq \depi{\lambda}(X_1,X_2).
$$
\end{lemma}
\begin{proof}
Note that, for $\lambda \in [0,1]$,
$$
\Yl = \sqrt{t}\bigl(\Xl\bigr) + \sqrt{1-t}Z,
$$
where $Z$ is a standard Gaussian, independent of $X_1$ and $X_2$.
By the chain rule for differential entropy 
\be \label{chaindoubl}
h(\Xl|\Yl) = h(\Xl) + h(\sqrt{1-t}Z) - h(\Yl).
\ee

On the other hand, using that conditioning reduces entropy and the EPI 
in dimension $d$, we have
\begin{align} 
\nonumber
&h(\Xl|\Yl) \\
\nonumber
&\geq h(\Xl|\tilde{X}_1^t,\tilde{X}_2^t) \\
\nonumber
&\geq \lambda h(X_1|\tilde{X}_1^t) + (1-\lambda)h(X_2|\tilde{X}_2^t) \\ 
\nonumber
&= \lambda h(X_1) + \lambda h(\sqrt{1-t}Z_1) + (1-\lambda)h(X_2) +(1-\lambda)h(\sqrt{1-t}Z_2) -\lambda h(\tilde{X}_1^t) - (1-\lambda)h(\tilde{X}_2^t) \\ 
\label{lastexprcondepi}
&= \lambda h(X_1) +  h(\sqrt{1-t}Z) + (1-\lambda)h(X_2) -\lambda h(\tilde{X}_1^t) - (1-\lambda)h(\tilde{X}_2^t),
\end{align}
where $Z_1,Z_2$ are independent standard Gaussians. 
Combining \eqref{chaindoubl} and \eqref{lastexprcondepi} and rearranging gives the result. 
\end{proof}
\begin{remark}
The proof of Lemma \ref{deficitdecrlemma} can easily be generalized to $n$ random variables as in \cite{carlen:91}, using the general form of the EPI.
\end{remark}

    The argument in the proof of Lemma \ref{deficitdecrlemma} is reminiscent of the argument made in \cite{green:25} to show that the (discrete) {\em doubling} does not increase under group homomorphisms. 

Before giving the proof of Theorem \ref{qualitativestabilityTh}, we recall the definition of the L{\'e}vy metric between two distribution functions $F$ and $G$
on $\R$:
\begin{align*}
&d_{\mathrm{L}}(F,G) := \inf{\Big\{\varepsilon>0: 
 F(x-\varepsilon) -\varepsilon \leq G(x) \leq
F(x+\varepsilon)+ \varepsilon, \text{ for all }x\in\R}\Big\}.
\end{align*}
When $F,G$ are the cumulative distribution functions of random variables $X$ and $Y$ respectively,
we write $d_{\mathrm{L}}(X,Y)=d_{\mathrm{L}}(F,G)$.
We will only make use of the fact that $d_{\mathrm{L}}$ metrizes weak convergence in $\R$. 

\begin{proof}[Proof of Theorem \ref{qualitativestabilityTh}]
We argue by contradiction. Suppose that the result is not true, i.e., \eqref{epsdeltaeq} does not hold. Let $\{(X_n,Y_n)\}_{n\geq1}$ be any joint sequence such that $X_n$ and $Y_n$ are independent and $\depi{\lambda}(X_n,Y_n) \to 0$.
Note first that, because of the uniformly-bounded $k$th moment assumption, 
the laws of $\{X_n\}$ and $\{Y_n\}$ form two tight sequences, so each subsequence of either has a convergent subsequence. 

We may thus extract a subsequence, along which both $X_n$ and $Y_n$ converge weakly. 
Assume without loss of generality that the common sequence is $\{(X_j,Y_j)\}_{j\geq1}$ and we have $X_{j}\stackrel{d}{\to}X_{\infty}$ and $Y_j \stackrel{d}{\to}Y_{\infty}$.

By assumption $\depi{\lambda}(X_j,Y_j) \to 0$, but the laws of either $X_j$ or $Y_j$ have bounded below L{\'e}vy distance from all Gaussians infinitely often, that is, there is a $c>0$ such that for infinitely many $j \geq 1$,
\be \label{contradiction}
\max\{d_{\mathrm{L}}(X_j,G_1),d_{\mathrm{L}}(Y_j,G_2)\} > c >0
\ee
for all Gaussians $G_1$ and $G_2$ that have the same variance.

Now consider the perturbed random variables $\tilde{X_j}$ and $\tilde{Y}_j$ as defined in \eqref{pertdef} for $t=\frac{1}{2}$ (any $t>0$ works), where we omit the superscript for simplicity. 
The density of $\tilde{X_j}$ is 
the expectation of a continuous, bounded function (the Gaussian density) of $X_j$ and therefore converges pointwise to the density of 
$\tilde{X}_{\infty}$, and similarly for the density of $\tilde{Y}_j$. 

In addition, the densities of $\tilde{X}_j,\tilde{Y}_j,\tilde{X}_{\infty},\tilde{Y}_{\infty}$ are all bounded above $\frac{1}{\sqrt{\pi}}$.
Moreover, by 
the uniformly-bounded $k$th moment assumption we have
$$\mathbb{E}|\tilde{X}_j|^k,\mathbb{E}|\tilde{Y}_j|^{k} \leq C_k,
$$
for some finite constant $C_k$ depending on the uniform
bound on $\mathbb{E}|X_n|^k$ and $\mathbb{E}|Y_n|^k$
and on $k$.
By considering the truncations 
$$
\mathbb{E}\Bigl(|\tilde{X}_{\infty}|^k; |\tilde{X}_{\infty}| \leq M\Bigr) = \lim_{j \to \infty}\mathbb{E}\Bigl(|\tilde{X_j}|^k; |\tilde{X}_j|\leq M \Bigr) \leq C_k,
$$
and similarly for $\tilde{Y}$, we see, after letting $M\to \infty,$ that $\tilde{X}_{\infty},\tilde{Y}_{\infty}$ satisfy the same $k$-moment condition.

Combining the previous two properties with the fact the densities of $\tilde{X}_{j}$ and $\tilde{Y}_{j}$ converge pointwise and using \cite[Theorem 1]{hero:04}, we conclude that 
\be \label{convergenceofhX}
h(\tilde{X}_j) \to h(\tilde{X}_{\infty}), \quad h(\tilde{Y}_j) \to h(\tilde{Y}_{\infty}).
\ee
A similar argument shows that 
\be \label{convergenceofhl}
h(\sqrt{\lambda}\tilde{X}_j + \sqrt{1-\lambda}\tilde{Y}_j) \to h(\sqrt{\lambda}\tilde{X}_{\infty} + \sqrt{1-\lambda}\tilde{Y}_{\infty}).
\ee

By Lemma \ref{deficitdecrlemma},
$$
0 \leq \depi{\lambda}(\tilde{X}_{j},\tilde{Y}_{j}) \leq \depi{\lambda}(X_{j},Y_j) \to 0.
$$
But by \eqref{convergenceofhX} and \eqref{convergenceofhl} we must have 
$$
\depi{\lambda}(\tilde{X}_{\infty},\tilde{Y}_{\infty}) = 0.
$$

Now, the assumption \eqref{contradiction} implies that $X_{\infty}$ is 
not Gaussian, or $Y_{\infty}$ is not Gaussian,
or they have different variances. By Cram{\'e}r's theorem, this implies 
that $\tilde{X}_{\infty}$ is not Gaussian, or $\tilde{Y}_{\infty}$ is not 
Gaussian, or they have different variances, which contradicts the 
equality case in the EPI. 

This establishes the second, equivalent formulation of the theorem,
and completes the proof.
\end{proof}

\begin{remark}
As can be seen from the proof, even if one is willing to fix the variances of $X_n$, and $Y_n$, these may well be different from that of the limit, as is the case in the counterexample of \cite{courtade:18}, where $X_n$ are Gaussian mixtures with variance $1,$ whereas $X_{\infty}$ is a Gaussian with variance $\frac{1}{2}$. As already mentioned, the latter example satisfies the assumption of Theorem \ref{qualitativestabilityTh} because of the fixed, finite variance.
\end{remark}
\begin{remark}
The proof of Theorem \ref{qualitativestabilityTh} can easily be extended to random vectors on $\R^d$ using the L{\'e}vy-Prokhorov metric, but then one has to invoke the equality case for the EPI in $\R^d$.
\end{remark}

\section{Proof of stability in Tao's EPI} \label{taosection}

\subsection{Log-concavity preliminaries}

A probability mass function (p.m.f.) 
$p : \mathbb{Z} \to \mathbb{R}_+$ is {\em log-concave} if 
$$
p(k)^2 \geq p(k-1)p(k+1), \quad \text{for all } k\in \mathbb{Z}. 
$$
It was shown in \cite{bobkov:22b} that a p.m.f.\ $p$ on the integers is log-concave if and only if there exists a continuous log-concave function $f: \mathbb{R}\to \mathbb{R}_+$ (the continuous, piecewise linear extension of $p$) such that 
\begin{equation*} 
f(k) = p(k),\quad \text{for all } k \text{ in } \mathbb{Z}.
\end{equation*}

We first prove a concentration lemma for continuous, log-concave functions. 
\begin{lemma} \label{lemmamatthieu}

Let  \( f:\mathbb{R} \to \mathbb{R}_{+} \) be a log-concave, integrable function. Denote $\mu = \frac{\int_{\mathbb{R}}{xf(x)dx}}{\int_{\mathbb{R}}f(x)dx}$, $\mu_- = \max\{-\mu,0\}$ and $\mu_+ = \max{\{\mu,0\}}$. Then
for every \( x \geq x_0 \geq \frac{3 \int_{\mathbb{R}} f(x)dx}{f(\mu)} + \mu_+ \),
\be
f(x) \leq f(x_0) 2^{-\frac{x-\mu}{x_0-\mu} + 1}.
\label{concf}
\ee
Analogously, the exact same bound
also holds for every \( x \leq x_0 \leq -\frac{3 \int_{\R} f(x)dx}{f(\mu)} 
- \mu_{-}\).
\end{lemma}
\begin{proof}
Write $x_{\mathrm{max}}$ for the mode of $f$, or one of its modes if
it is not unique. Assume without loss of generality that $x_{\mathrm{max}} \geq \mu$ (otherwise, consider $g(x) = f(-x)$). 

We will only prove the positive case, since the negative case is similar. Proceeding similarly to the proof of \cite[Lemma 22]{fradelizi:arxiv24} and noting that 
\begin{equation} \label{modecontrol}
\int_{\R}{f(x)dx} \geq \int_{\mu}^{\infty} f(x) dx \geq \int_{\mu}^{x_{\mathrm{max}}}f(\mu) = f(\mu)(x_{\max}-\mu),
\end{equation}
we have
\begin{equation} \label{intf03}
2 \int_{\R}^{} \frac{f(x)}{f(\mu)}  \, dx + x_{\max} \leq 3 \frac{\int_{\R} f(x)dx}{f(\mu)} + \mu \leq  3 \frac{\int_{\R} f(x)dx}{f(\mu)} + \mu_+.
\end{equation}
Thus, using the assumption on $x_0$ and \eqref{intf03}, 
\[
\int_{\R}^{} f(x)  dx \geq \int_{x_{\max}}^{x_0} f(x) \, dx \geq f(x_0) (x_0 - x_{\max}) \geq f(x_0) \left( 2\int_{\R}^{} \frac{f(x)}{f(\mu)} \, dx + x_{\max} - x_{\max} \right),
\]
or,
\begin{equation} \label{half}
f(x_0) \leq \frac{f(\mu)}{2}.
\end{equation}

Now, using log-concavity and writing \( x_0 = \frac{x_0-\mu}{x-\mu} x + \left( 1 - \frac{x_0-\mu}{x-\mu} \right) \mu \), where by assumption $0\leq \frac{x_0-\mu}{x-\mu} \leq 1$, we have
\[
\log{f(x_0)} \geq \frac{x_0-\mu}{x-\mu}\log{f(x)} + (1-\frac{x_0-\mu}{x-\mu})\log{f(\mu)},
\]
or,
\[
f(x)^{\frac{x_0-\mu}{x-\mu}} \leq \frac{f(x_0)}{f(\mu)^{-\frac{x_0-\mu}{x-\mu}+1}},
\]
hence,
\[
f(x) \leq f(x_0)\Bigl(\frac{f(x_0)}{f(\mu)}\Bigr)^{\frac{x-\mu}{x_0-\mu}-1}.
\]
This combined with \eqref{half} gives \eqref{concf}.
\end{proof}

The following lemma shows, via finite bounds,
that a discrete log-concave p.m.f.\
eventually decays exponentially.
\begin{lemma} \label{exponentialp}
 Let $p:\mathbb{Z} \to \mathbb{R}_+$ be a discrete log-concave p.m.f., with mean $\mu_p := \sum_{k \in \mathbb{Z}}p(k)k$ and variance $\sigma^2.$ Assume $\sigma \geq 2.$ Then for all positive integers $k,m$ such that $k \geq m \geq 49\sigma +2\mu_p + 8$, we have,
$$
p(k) \leq  p(m)2^{-\frac{k-\mu}{m-\mu}+1},
$$
for some $\mu \in \mathbb{R}$ with 
\be \label{muassumptionbounds}
|\mu| \leq 8 + 2|\mu_p|.
\ee
Analogously, 
the exact same bound holds 
for all negative integers $k \leq m \leq -(49\sigma +2\mu_p + 8)$,
for some $\mu$ satisfying \eqref{muassumptionbounds}.
\end{lemma}
\begin{proof}
By \cite[Proposition 5.1]{bobkov:22b} there exists a continuous log-concave function $f$ such that $f(k) = p(k)$ for all integers $k$.

We prove the right-tail bound. 
By inequality (50) in the proof of \cite[Prop. 24]{fradelizi:arxiv24}, which holds for any integrable, log-concave function $f$ (not necessarily 
a density), we have 
$$
\Big|\int_{\mathbb{R}}{f(x)dx} - \sum_k {p(k)}\Big| \leq \max_{\mathbb{Z}}{p(k)} \leq \frac{1}{\sigma} \leq \frac{1}{2},
$$
where we used \cite[Theorem 1.1]{bobkov:22b} and the assumption $\sigma \geq 2$.

Write $\mu = \frac{\int_{\R}{xf(x)dx}}{\int_{\R}{f(x)dx}}$. By \cite[Theorem 4]{fradelizi:97}
\begin{equation*} 
f(\mu)\geq \frac{\max_{\mathbb{R}}{f(x)}}{e} \geq \frac{\max_{\mathbb{Z}}{f(k)}}{e}.
\end{equation*}
In addition, by the one-dimensional part of \cite[Proposition 27]{fradelizi:arxiv24},
\begin{equation*}
    \left|\mu-\frac{\mu_p}{\int_{\R}f(x)dx}\right|\leq (e+1)\frac{\sum_{k\in \mathbb{Z}} f(k)}{\int_{\R}{f(x)dx}} \leq \frac{4}{1-\frac{1}{\sigma}} \leq 8.
\end{equation*} 
The latter implies 
\be \label{meanbounds}
-\frac{2}{3}|\mu_p|-8 \leq -\frac{|\mu_p|}{\int_{\R}f(x)dx} -8\leq \mu \leq 8+\frac{|\mu_p|}{\int_{\R}f(x)dx} \leq 8+2|\mu_p|.
\ee

Combining the above bounds we have 
\be
\frac{3\int_{\R}{f(x)dx}}{f(\mu)} + \mu_+ \leq \frac{9}{2f(\mu)} + |\mu| \leq \frac{9e}{2\max_{k}p(k)} + 8 +2|\mu_p| \leq 49\sigma + 8 + 2\mu_p,
\label{combined}
\ee
where we also used the fact that $\max_kp(k) \geq \frac{1}{4 \sigma}$ for $\sigma \geq 1$ by \cite[Theorem 1.1]{bobkov:22b}.
Therefore, the assumption on $m$ implies that the first assumption of Lemma \ref{lemmamatthieu} is satisfied for $x_0 = m$ and $x = k$. Applying that lemma we obtain,
$$
p(k) \leq  p(m)2^{-\frac{k-\mu}{m-\mu}+1},
$$
where $\mu$ satisfies the bounds \eqref{meanbounds} and therefore also the claimed bound \eqref{muassumptionbounds}. 

The left tail can be bounded in an analogous manner. 
\end{proof}

\subsection{Cheeger constant, Poincar{\'e} constant and stability} \label{cheegersec}

Here we show that the Poincar{\'e} constant of a discrete log-concave random variable is bounded by a constant times its variance, and use this to prove Theorem \ref{mainthdiscrstable}. We recall the definition of the Poincar{\'e} constant of a random variable $X$, denoted by $C_P(X)$: 
$$
C_P(X) = \sup{\frac{\Var{(g(X))}}{\mathbb{E}\bigl[\bigl(g^{\prime}(X)\bigr)^2\bigr]}},
$$
where the supremum is over all continuously differentiable functions $g$. 
We also write $C_P(\nu) = C_P(X)$ when $X$ has law $\nu.$ 

The isoperimetric constant of a separable probability measure $\nu$ on a metric space $(\mathcal{X},d)$ is defined as 
$$
\Is(\nu) := \inf \frac{\nu^+(A)}{\min{\{\nu(A),1-\nu(A)\}}},
$$
where the infimum ranges over all Borel sets $A\subset \mathcal{X}$ with 
$0<\nu(A)<1$, $\nu^+(A) := \liminf_{h\to0^+}\frac{\nu(A^h)-\nu(A)}{h}$ and 
$A^h$ denotes the open neighborhood of $A$ with radius $h$ with respect to $d$.
It was introduced by Cheeger \cite{cheeger:71} in the 
context of Riemannian geometry.

We are interested in the special case where $(\mathcal{X},d) = (\R,|\cdot|).$ If $F(x)$ is the distribution function of a probability measure $\nu$ with density $f_\nu$ with respect to the Lebesgue measure, then we have the following characterization of the isoperimetric constant due to Bobkov and Houdr{\'e} \cite[Theorem 1.3]{bobkov:97}:
\be \label{Isexpr}
\Is(\nu) = \essinf_{a<x<b}\;\frac{f_\nu(x)}{\min{\{F(x),1-F(x)}\}},
\ee
where $a = \inf{\{x:F(x) > 0\}}$, $b = \sup{\{x:F(x)<1\}}$,
and the essential infimum is with respect to Lebesgue measure
on $(a,b)$.

Cheeger's inequality \cite{cheeger:71} (see also \cite[Eq. (3.8)]{bobkov:97}) states that,
\be \label{cheegerineq}
C_P(\nu) \leq \frac{4}{\Is^2(\nu)}.
\ee
Therefore, any lower bound on the isoperimetric constant gives
an upper bound on the Poincar{\'e} constant. 

\begin{proof}[Proof of Proposition \ref{poincareprop}]
Let $\f$ be the (piecewise constant) density of $X+U$,
$\F$ be its distribution function, and $\nu$ its law.
In view of \eqref{Isexpr} and \eqref{cheegerineq}, our goal is to show that, 
\be \label{goalpoincare}
\min{\{\F(x),1-\F(x)\}} \leq C\sigma \f(x) \quad \text{for Lebesgue-almost every } x\in \R,
\ee
for an appropriate absolute constant $C$,
where $\sigma^2$ is the variance of $X$.

Write $p$ for the p.m.f.\ of $X$ and $\mu_p$ for its mean. By shifting $X$ by an integer if necessary, we may assume that $\mu_p \in [0,1]$. As before, by \cite[Proposition 5.1]{bobkov:22b}, there exists a continuous function $f$ which is log-concave and $f(k) = p(k)$ for all $k\in\Z$.
Using \eqref{modecontrol} for the mode $x_m$ of $f$,
noting that an analogous lower bound on the mode may be proved 
in the same way in the case $x_m \leq \mu$, 
and applying \eqref{combined}, we obtain that

\begin{equation*}
|x_m| \leq \frac{\int_{\R}{f(x)dx}}{f(\mu)} + |\mu| \leq \frac{3}{2f(\mu)} + |\mu| \leq \frac{3e}{2\max_{k}p(k)} + 8 +2|\mu_p|.
\end{equation*}
To see that the first inequality above follows 
from \eqref{modecontrol}, note that, either,
$$
\mu \leq x_m \leq \frac{\int_{\R}{f(x)dx}}{f(\mu)} + \mu_+,
$$
or,
$$
\mu \geq x_m \geq -\frac{\int_{\R}{f(x)dx}}{f(\mu)} - \mu_-.
$$
Hence, using again the fact that $\max_kp(k) \geq \frac{1}{4\sigma}$ 
for $\sigma \geq 1$, we see that 
\begin{equation*}
|x_m| \leq 18\sigma+10,
\end{equation*}
where we also used the fact that $\mu_p \leq 1$. 
By unimodality, the mode $\xmp$ of $p$, satisfies 
$
x_m -1 \leq \xmp \leq x_m+1,
$
and, therefore, 
\be \label{modepbound}
|\xmp| \leq 18\sigma+11.
\ee

We will show that, for $x \geq \xmp$ we have $1-\F(x) \leq C\sigma \f(x)$,
while for $x < \xmp$ we have $\F(x) \leq C\sigma\f(x)$.

Consider the first case, $x \geq \xmp$, and let $m = \lfloor x\rfloor$. 
With this notation, $\f(x) = p(m)$. For 
$$m \geq 49\sigma+10, 
$$
we will use Lemma \ref{exponentialp}, which is applicable since we have assumed $\mu_p \in [0,1]$. We have
$$1-\F(x) 
= \int_{x}^{\infty}\f(y)dy 
\leq \sum_{k=m}^{\infty}p(k)
= p(m)\sum_{k=m}^{\infty}\frac{p(k)}{p(m)}.
$$
Recalling that 
discrete log-concavity
says that $\frac{p(k+1)}{p(k)}$ is non-increasing in $k$,
we obtain
\begin{align}
1-\F(x) 
\nonumber
&\leq p(m)\sum_{k=m}^{\infty}{\Bigl(\frac{p(m+1)}{p(m)}}\Bigr)^{k-m} \\ \label{pmceilingsbound}
&\leq p(m)\sum_{k=m}^{\infty}{\Bigl(\frac{p(\lceil 49\sigma+10\rceil+1)}{p(\lceil 49\sigma+10\rceil)}}\Bigr)^{k-m}.
\end{align}
Now, by Lemma \ref{exponentialp},  we have
\be
\frac{p(\lceil 49\sigma+10\rceil+1)}{p(\lceil 49\sigma+10\rceil)}\leq 2^{-\frac{1}{\lceil 49\sigma+10\rceil-\mu}} \leq 2^{-\frac{1}{49\sigma+21}},
\label{lgtails}
\ee
where $\mu$ is as in the lemma, and where we used the 
bound \eqref{muassumptionbounds} and the fact that $|\mu_p|\leq 1$.
Therefore, substituting \eqref{lgtails} into \eqref{pmceilingsbound},
a simple change of variables in the sum gives, 
\begin{align}
\nonumber
1-\F(x) &\leq \f(x)\sum_{k=0}^{\infty} 2^{-\frac{k}{\SI}} \\ \label{exponentialelementarybound}
&= \f(x)\frac{1}{1-2^{-\frac{1}{\SI}}}.
\end{align}
Using the elementary bound $1-x\geq e^{-2x} \geq 2^{-4x}$, 
for $x< \frac{\log{2}}{2}$, we can 
bound \eqref{exponentialelementarybound} further to obtain 
\begin{align*}
1-\F(x) &\leq \f(x) \frac{1}{1-e^{-\frac{1}{98\sigma + 42}}} \\
&\leq \f(x)\cdot (196 \sigma + 84) \\
&\leq 238  \sigma \cdot \f(x).
\end{align*}
On the other hand, if 
$$m < 49\sigma +10,$$
while $x\geq \xmp$, using \eqref{modepbound} and unimodality
we have,
\begin{align}
\nonumber
1-\F(x) &\leq \sum_{k = m}^{\lfloor 49\sigma+10 \rfloor}p(k) + \sum_{k=\lfloor 49\sigma+10 \rfloor}^{\infty}p(k) \\
\nonumber
&\leq \sum_{-18\sigma-11\leq k \leq 49\sigma + 10}p(m)+\sum_{k=\lfloor 49\sigma+10 \rfloor}^{\infty}p(k) \\ \label{secondtermrepeat}
&\leq 78 \sigma \cdot p(m) + 253\sigma \cdot p(\lfloor 49 \sigma +10\rfloor) \\
\nonumber
&\leq 331\sigma \cdot p(m) \\ \label{finalleftcdf}
&= 331\sigma \cdot \f(x),
\end{align}
where we bounded the second term in \eqref{secondtermrepeat} in the same way as in \eqref{pmceilingsbound}.

The second case, where $x < \xmp$, is treated 
exactly the same way: using the left-tail bounds of 
Lemma \ref{exponentialp}, we can show that
\be \label{finalrightcdf}
\F(x) \leq 331\sigma \cdot \f(x). 
\ee

The two bounds \eqref{finalleftcdf} and \eqref{finalrightcdf} 
give \eqref{goalpoincare}, as desired, with $C = 331$.
By Cheeger's inequality \eqref{cheegerineq} we conclude that:
$$
C_P(X+U) \leq 4\cdot 331^2 \sigma^2  = 438244 \cdot \sigma^2.
$$
\end{proof}

We are now ready to give the proof of Theorem \ref{mainthdiscrstable}: 
\begin{proof}[Proof of Theorem \ref{mainthdiscrstable}]
Let $Z$ be a Gaussian with the same mean and variance as $X+U,$ 
where $U$ is a continuous uniform on the unit interval,
independent of $X$.
Then by \cite[Theorem 4.1, (ii)]{KM:14}, i.e., equation \eqref{kontoyianniseq}, applied to $X+U,$ 
\begin{align*}
D(X+U \| Z) &\leq \Bigl( \frac{2 C_P(X+U)}{\sigma^2 + \frac{1}{12}} + 1\Bigr)\Bigl( h(X_1+ U_1 + X_2 + U_2) - h(X_1 + U_1) - \frac{1}{2}\log{2} \Bigr)\\
&\leq 876489\cdot\Bigl( h(X_1+ U_1 + X_2 + U_2) - h(X_1 + U_1) - \frac{1}{2}\log{2} \Bigr),
\end{align*}
where we used Proposition \ref{poincareprop}.

Now we invoke \cite[Theorem 1]{gavalakis:24}. 
The assumption $\sigma > 1547 > \frac{3^7}{\sqrt{2}}$ ensures that 
we may apply the finite bound in \cite[Eq. (11)]{gavalakis:24},
where the error term
in this case is at most 
$
C\frac{\log{\sigma}}{\sigma}
$
for some absolute constant $C$.
Thus, we have,
\be \label{contstabilitypoinc}
D(X+U \| Z) \leq C_1\Bigl( H(X_1+ X_2) - H(X_1) - \frac{1}{2}\log{2} + C_2\frac{\log{\sigma}}{\sigma} \Bigr),
\ee
for some absolute constants $C_1,C_2$.

By \cite[Theorem 2.1]{gavalakis:clt}, with $n=1$ and maximal span equal to 1,
we finally obtain,
\begin{align*}
D(X\|Z^{(\mathbb{Z})}) &\leq \frac{1}{2}\log{(2\pi e \sigma^2)} - H(X) + \frac{2}{\sigma}\\
&\leq \frac{1}{2}\log{2\pi e (\sigma^2+\frac{1}{12})} - h(X+U) + \frac{2}{\sigma} \\
&=D(X+U\|Z) + \frac{2}{\sigma}\\
&\leq C_1\Bigl( H(X_1+ X_2) - H(X_1) - \frac{1}{2}\log{2}\Bigr) + C^{\prime}_2\frac{\log{\sigma}}{\sigma},
\end{align*}
for some absolute constant $C_2^{\prime}$,
where the last inequality follows by \eqref{contstabilitypoinc}. 
\end{proof}

\appendix

\section*{Appendix: Proof of the EPI via de Bruin's identity} \label{appendix}

Here we present Stam's proof of the EPI \cite{stam:59},
in the form given in \cite{dembo-cover-thomas:91}, 
but justifying de Bruijn's identity and continuity of the entropy function 
using our results, which avoid the finite variance assumption.

Recall that Stam's inequality 
\cite{stam:59, blachman:65, dembo-cover-thomas:91},
\be \label{blachmanstam}
I(\sqrt{\lambda} X + \sqrt{1-\lambda} Y) \leq \lambda I(X) + (1-\lambda) I(Y),
\ee
holds for any pair of independent random variables $X,Y$ and all 
$\lambda\in(0,1)$, with the Fisher information defined as in~\eqref{FIdef}.
Moreover, if $I(X)$ and $I(Y)$ are finite, there is 
equality in \eqref{blachmanstam} if and only if $X,Y$ are Gaussian. 
In the proof below we will only use \eqref{blachmanstam} on perturbed 
random variables, which possess smooth densities, so existence of 
Fisher information will not be an issue. 

\begin{proof}[Proof of Theorem \ref{ourEPI}]
If $h(X) = -\infty,$ $h(Y) = -\infty$ or $h(\sqrt{\lambda}X+\sqrt{1-\lambda}Y) = \infty$, there is nothing to prove. So assume $h(X),h(Y) > -\infty$ and $h(\sqrt{\lambda}X+\sqrt{1-\lambda}Y)< \infty$. In particular, $X,Y$ must have densities with respect to the Lebesgue measure. 

Let \(Z_{1},Z_{2}\sim\mathcal{N}(0,1)\) be independent of \(X,Y\).
For \(t\in[0,1]\) consider the perturbations, as defined in \eqref{pertdef}, 
\begin{align*}
  \Xt &= \sqrt{t}X + \sqrt{1-t}\,Z_{1}, \\
   \Yt &= \sqrt{t}Y + \sqrt{1-t}\,Z_{2}.
\end{align*}
Now fix $\lambda \in (0,1)$ and consider $V_t = \sqrt{\lambda}\Xt + \sqrt{1-\lambda}\Yt$. Note that 

\be \label{Vtexpr}
V_t = \sqrt{t}V_1 + \sqrt{1-t}V_0,
\ee
where $V_0 \sim \mathcal{N}(0,1)$.

Define
\[
\Delta(t) := h(V_t) - \lambda h(X_t) - (1 - \lambda) h(Y_t) \quad t \in [0,1].
\]
Then,
\[
\Delta(1) = h(V_1) - \lambda h(\tilde{X}^1) - (1 - \lambda) h(\tilde{Y}^1) = h(\sqrt{\lambda} X + \sqrt{1 - \lambda} Y) - \lambda h(X) - (1 - \lambda) h(Y).
\]
We aim to prove that \( \Delta(1) \ge 0 \).
From de Bruijn's identity, Theorem \ref{ourdebruijn}, we have that $\Delta(t)$ is differentiable for all $t<1$ and therefore also continuous. Thus, 
\[
\Delta(0) = \lim_{t \downarrow 0} \Delta(t) = 0.
\]
Since, by Theorem \ref{entropycontTh}, 
$\Delta(1) = \lim_{t \uparrow 1} \Delta(t),$ it suffices to prove 
that \( \Delta^\prime(t) \ge 0 \) for all \( t \in (0,1) \). Note that,
since we have assumed $h(\sqrt{\lambda}X+\sqrt{1-\lambda}Y) < \infty$,
the 
assumptions of 
Theorem \ref{entropycontTh}
are satisfied by $\sqrt{\lambda}X$ and $\sqrt{1-\lambda}Y$.
Also, by an application of the submodularity inequality Lemma \ref{km14Lemma},
$\sqrt{\lambda}X + \sqrt{1-\lambda}Y$ satisfies the assumption of 
Theorem \ref{entropycontTh} as well. 

Let
\[
s(t) = \frac{1 - t}{t}, \qquad s^\prime(t) = -\frac{1}{t^2},
\]
and note that,
\[
\Xt = \sqrt{t} X + \sqrt{1 - t} Z_1 = \sqrt{t} \left( X + \sqrt{s(t)} Z_1 \right).
\]
Since,
\[
h(\Xt) = \frac{1}{2} \log t + h\left(X + \sqrt{s(t)} Z_1\right),
\]
we have, by de Bruijn's identity and the scaling property $I(aX) = \frac{1}{a^2}I(X),$
\[
\frac{d}{dt} h(\Xt) = \frac{1}{2t} + \frac{1}{2} s^\prime(t) I\left(X + \sqrt{s(t)} Z_1\right) = \frac{1}{2t} - \frac{1}{2t} I(X_t).
\]
Similarly, for \( \Yt \) and \( V_t \), recalling the expression \eqref{Vtexpr}, we get
\[
\frac{d}{dt} h(\Yt) = \frac{1}{2t} - \frac{1}{2t} I(\Yt), \qquad \frac{d}{dt} h(V_t) = \frac{1}{2t} - \frac{1}{2t} I(V_t).
\]

Thus, the derivative of \( \Delta(t) \) becomes
\[
\Delta^\prime(t) = \frac{1}{2t} \left[ \lambda I(\Xt) + (1 - \lambda) I(\Yt) - I(V_t) \right],
\]
and from Stam's
inequality \eqref{blachmanstam} we have 
$ \Delta^\prime(t) \geq 0$, for all $t \in (0,1)$,
completing the proof of the EPI.

Now, suppose
there is equality in the EPI. Since, by assumption $h(X),h(Y) < \infty,$ 
we must have that $h(\sqrt{\lambda}X+\sqrt{1-\lambda}Y) < \infty$. Moreover, we must have $\Delta^{\prime}(t) = 0$ for every $t \in [0,1]$, which by the equality case in the Fisher information inequality implies that $\Xt$ and $\Yt$ are Gaussian. Cram{\'e}r's characterization of the Gaussian distribution shows that $X$ and $Y$ must be Gaussian as well. Since by assumption $h(X),h(Y)$ are finite and $X,Y$ are Gaussian, their variances must be finite as well. Moreover, if $X$ and $Y$ are Gaussian and we have equality in \eqref{EPI}, we have 
$$
\frac{1}{2}\log{\Bigl({{\lambda}\mathrm{Var}(X) + {(1-\lambda)}\mathrm{Var}(Y)}\Bigr)} = \frac{\lambda}{2}\log{\mathrm{Var}(X)} + \frac{(1-\lambda)}{2}\log{\mathrm{Var}(Y)},
$$ 
which implies that $\mathrm{Var}(X) = \mathrm{Var}(Y)$.
\end{proof}

\bibliographystyle{plain}
\bibliography{ik}

\end{document}